\documentclass[10pt]{article}
\usepackage{amssymb,latexsym,amsmath,epsfig,amsthm} 

\makeatletter

\renewcommand\section{\@startsection {section}{1}{\z@}
{-30pt \@plus -1ex \@minus -.2ex}
{2.3ex \@plus.2ex}
{\normalfont\normalsize\bfseries\boldmath}}

\renewcommand\subsection{\@startsection{subsection}{2}{\z@}
{-3.25ex\@plus -1ex \@minus -.2ex}
{1.5ex \@plus .2ex}
{\normalfont\normalsize\bfseries\boldmath}}

\renewcommand{\@seccntformat}[1]{\csname the#1\endcsname. }

\makeatother

\newtheorem{theorem}{Theorem}
\newtheorem{lemma}{Lemma}

\theoremstyle{definition}

\newtheorem{conjecture}{Conjecture}

\begin{document}
																												
\begin{center}
\uppercase{\bf \boldmath A Short Proof of a Conjecture Regarding Quadratic Representations of Practical Numbers}
\vskip 20pt
{\bf Ting Hon Stanford Li}\\
\end{center}

\centerline{\bf Abstract}
\noindent
A positive integer $n$ is called a practical number if every positive integer less than or equal to $n$ can be expressed as a sum of distinct positive divisors of $n$. In this paper, we study an open conjecture proposed by Wang and Sun concerning the quadratic representations of practical numbers. Specifically, we provide a short proof of the second part of the conjecture, demonstrating that for any positive integers $b$ and $c$ with $2 \nmid b$ and $2 \mid c$, there exists an integer $n$ satisfying $1 < n \le \max\{b, c\}$ such that $n^2 + bn + c$ is a practical number. Combined with the work of Somu, Li, and Kukla, this completely settles the conjecture of Wang and Sun.

\pagestyle{myheadings}
\thispagestyle{empty}
\baselineskip=12.875pt
\vskip 30pt

	\section{Introduction}

A positive integer $n$ is called a practical number if every positive integer less than or equal to $n$ can be expressed as a sum of distinct positive divisors of $n$. This concept was introduced by Srinivasan \cite{Srinivasan} and systematically studied by Stewart \cite{Stewart}.

Recently, the problem of quadratic representations of practical numbers has been studied by various researchers. For instance, Wang and Sun investigated the existence of practical numbers represented by various polynomial forms, including quadratic forms \cite{Wang}. In their paper, Wang and Sun also proposed a conjecture regarding these quadratic forms.

\begin{conjecture}\label{main}
Let $a, b, c$ be positive integers with $2 \nmid ab$ and $2 \mid c$. Then:
\begin{enumerate}
 \item there are infinitely many positive integers $n$ such that $an^2 + bn + c$ is practical; 
 \item in the case $a = 1$, there is an integer $n$ with $1 < n \le \max\{b, c\}$ such that $n^2 + bn + c$ is practical.
\end{enumerate}
\end{conjecture}

Somu, Li, and Kukla proved the first part of the conjecture \cite{Somu}. In this article, we provide a short proof for the second part of Wang and Sun's conjecture.

Throughout this paper, let $f(n) = n^2 + bn + c$, where $b, c \in \mathbb{Z}^+$ satisfy $2 \nmid b$ and $2 \mid c$. Furthermore, we define $m = \lfloor \log_2(\max\{b, c\}) \rfloor$.

	\section{Proof of the Main Result}
    
In order to prove the second part of Conjecture~\ref{main}, we rely on Hensel's lemma and a well-known property of practical numbers, which are presented below:

\begin{lemma}[{Hensel's Lemma, \cite[Theorem~2.23]{Niven}}]\label{hensel}
Suppose that $f(x)$ is a polynomial with integral coefficients. If $f(a) \equiv 0 \pmod{p^j}$ and $f'(a) \not\equiv 0 \pmod{p}$, then there is a unique $t \pmod{p}$ such that $f(a +tp^j) \equiv 0 \pmod{p^{j+1}}$.
\end{lemma}

\begin{lemma}[{\cite[Lemma~1]{Melfi}}]\label{practical_numbers2}
Let $\sigma(k)$ denote the sum of the positive divisors of $k$. If $k$ is a practical number and $n$ is an integer satisfying $1 \leq n \leq \sigma(k) + 1$, then the product $kn$ is also a practical number. In particular, $kn$ is practical for any integer $n$ such that $1 \leq n \leq 2k$.
\end{lemma}

We then present two more lemmas which are useful for our proof.

\begin{lemma}\label{lemma1}
For all positive integers $m$, there exist exactly two distinct roots $n_1, n_2 \in \{0, 1, \dots, 2^m-1\}$ such that $f(n_1) \equiv 0 \pmod{2^m}$ and $f(n_2) \equiv 0 \pmod{2^m}$.
\end{lemma}

\begin{proof}
Since $b$ is an odd integer and $c$ is an even integer, the polynomial satisfies $f(x) \equiv x^2 + x \equiv 0 \pmod 2$. This congruence equation has two distinct roots, $0$ and $1$, modulo $2$. Hence, the lemma is true for $m=1$.

The derivative $f'(x) = 2x + b \equiv b \equiv 1 \not\equiv 0\pmod 2$ for all integers $x$. Therefore, Hensel's lemma (Lemma~\ref{hensel}) guarantees that both roots lift uniquely to exactly two distinct roots modulo $2^m$ for any $m \geq 2$. Hence, there exist two distinct roots $n_1, n_2 \in \{0, 1, \dots, 2^m-1\}$ such that $f(n_1) \equiv 0 \pmod{2^m}$ and $f(n_2) \equiv 0 \pmod{2^m}$ for $m \geq 2$.
\end{proof}

\begin{lemma}\label{lemma2}
Let $n_1, n_2$ be the two distinct roots $\in \{0, 1, \dots, 2^m-1\}$ for the congruence equation $f(n) \equiv 0 \pmod{2^m}$ such that $n_1 < n_2$, then:
\[
f(n_1) < 2^{2m+1}.
\]
\end{lemma}

\begin{proof}
Let $b = q \cdot 2^m + b_0$, where $q$ is a non-negative integer and $0 \leq b_0 < 2^m$. Note that since $b$ is odd, $b_0$ is odd. As a result, $b_0 \geq 1$.

We first show that 
\[
n_1 \leq 2^m - \frac{b_0+1}{2}.
\]

Since $n_1$ and $n_2$ are both roots of the congruence equation $f(n) \equiv 0 \pmod{2^m}$, we have $n_1^2 + bn_1 + c \equiv 0 \pmod{2^m}$ and $n_2^2 + bn_2 + c \equiv 0 \pmod{2^m}$. Subtracting these two congruences yields $(n_1^2 - n_2^2) + b(n_1 - n_2) \equiv 0 \pmod{2^m}$, which can be further simplified as:
\[
(n_1 - n_2)(n_1 + n_2 + b) \equiv 0 \pmod{2^m}.
\]

In the proof of Lemma~\ref{lemma1}, we show that the roots of $f(n) \equiv 0 \pmod 2$ are $0$ and $1$, where one of them is odd and the other one is even. Note that in the proof of Lemma~\ref{lemma1}, the lifting of roots by Hensel's lemma (Lemma~\ref{hensel}) does not change the parity of the roots. This implies that one of the lifted roots $n_1, n_2$ must be even and the other must be odd. Consequently, their difference $n_1 - n_2$ is an odd integer, and is therefore coprime to $2^m$. Thus,
\[
n_1 + n_2 + b \equiv 0 \pmod{2^m} \implies n_1 + n_2 \equiv -b \pmod{2^m}.
\]
Substituting the expansion $b = q \cdot 2^m + b_0$ into the congruence yields:
\begin{equation}\label{eqn}
    n_1 + n_2 \equiv -b_0 \equiv 2^m - b_0 \pmod{2^m}.
\end{equation}

Since $0 \leq n_1 < n_2 \leq 2^m - 1$, we have $0 < n_1 + n_2 < 2^{m+1} - 1$. There are only two possible integer values of $n_1 + n_2$ that can satisfy the linear congruence condition: $2^m - b_0$ and $2^{m+1} - b_0$.

In both scenarios, $n_1 + n_2 \leq 2^{m+1} - b_0$, and so we have $n_1 < \frac{2^{m+1} - b_0}{2} = 2^m - \frac{b_0}{2}$. But since $b_0$ is odd, the inequality can be rewritten as
\[
n_1 \leq 2^m - \frac{b_0+1}{2}.
\]

Since $f(x) = x^2 + bx + c$ is strictly increasing for all $x \geq 0$, the maximum possible value of $f(n_1)$ is $f(2^m - \frac{b_0+1}{2})$. Substituting $b = q \cdot 2^m + b_0$ into $f(2^m - \frac{b_0+1}{2})$, we define the function $G(b_0)$ as:
\[
G(b_0) = \left(2^m - \frac{b_0 + 1}{2}\right)^2 + (q \cdot 2^m + b_0)\left(2^m - \frac{b_0 + 1}{2}\right) + c.
\]
We take its first derivative with respect to the variable $b_0$:
\[
\frac{dG}{db_0} = -\frac{q \cdot 2^m + b_0}{2}
\]
Since $q \geq 0$ and $b_0 \geq 1$, this derivative is strictly negative ($\frac{dG}{db_0} < 0$). So the function $G(b_0)$ is strictly decreasing. Consequently, the absolute maximum value of $G(b_0)$ must occur at $b_0 = 1$, so
\[
f(n_1) \leq f(2^m - \frac{b_0+1}{2}) \leq \left(2^m - 1\right)^2 + (q \cdot 2^m + 1)\left(2^m - 1\right) + c .
\]

By definition, $m = \lfloor \log_2(\max\{b, c\}) \rfloor$, which implies $b, c < 2^{m+1}$. Since $b = q \cdot 2^m + b_0 < 2^{m+1}$, we have $q \leq 1$. Moreover, because $c$ is an even integer strictly less than $2^{m+1}$, we have $c \leq 2^{m+1} - 2$. Using the upper bounds $q \leq 1$ and $c \leq 2^{m+1} - 2$, we conclude that:
\[
f(n_1) \leq 2^{2m+1} - 2 < 2^{2m+1}.
\]
\end{proof}

We then prove our main result.

\begin{theorem}
Let $b$,$c$ be positive integers such that $b$ is odd and $c$ is even. Then there is an integer $n$ with $1 < n \leq \max\{b,c\}$ such that $n^2 + bn +c$ is practical.
\end{theorem}

\begin{proof}
Let $f(n) = n^2 + bn +c$. We define $n_1, n_2$ as in Lemma~\ref{lemma2}.

\vskip 5pt\noindent {\tt Case 1:} $n_1 > 1$.

We claim that $n_1$ is an integer which satisfies the criteria of the theorem. Note that $n_1 \leq 2^m - 1 < 2^{\lfloor \log_2(\max\{b, c\}) \rfloor} \leq \max\{b,c\}$. Since $f(n_1) \equiv 0 \pmod{2^m}$ and $f(n_1) < 2^{2m+1}$ (By Lemma~\ref{lemma2}), $f(n_1)$ can be expressed in the form $k2^m$, where $k$ is a positive integer such that $k < \frac{2^{2m+1}}{2^m}=2^{m+1}$. Since $2^m$ is practical and $k \le 2(2^m)$, by Lemma~\ref{practical_numbers2}, $f(n_1)=k2^m$ is also practical. This completes the proof for this case.

\vskip 5pt\noindent {\tt Case 2:} $n_1 \leq 1$ and $n_2>1$. 

We define $q$ and $b_0$ as in the proof of Lemma~\ref{lemma2}. We claim that $n_2$ is an integer which satisfies the criteria of the theorem. Since $n_1 \le 1$ and $n_2 \le 2^m - 1$, their sum satisfies $1 \le n_1 + n_2 \le 2^m$. The only value in this range satisfying the congruence condition $n_1 + n_2 \equiv -b_0 \pmod{2^m}$ (Equation~\ref{eqn}) is $n_1 + n_2 = 2^m - b_0$. Since $n_1 \ge 0$, we have $n_2 \le 2^m - b_0 < \max\{b,c\}$. Moreover, since $f(x) = x^2 + bx + c$ is strictly increasing for all $x \geq 0$, we have $f(n_2) \leq f(2^m - b_0)$. Substituting $b = q \cdot 2^m + b_0$:
\[
f(n_2) \le (2^m - b_0)^2 + (q \cdot 2^m + b_0)(2^m - b_0) + c.
\]
Expanding and simplifying the right-hand side gives:
\[
f(n_2) \le 2^{2m} - b_0 \cdot 2^m + q \cdot 2^{2m} - q \cdot b_0 \cdot 2^m + c = 2^m[2^m(1+q) - b_0(1+q)] + c.
\]
By our definition of $m$, we have $\max\{b,c\} < 2^{m+1}$, so $q \le 1$ and $c \le 2^{m+1} - 2$. Together with $b_0 \geq 1$, we have:
\[
f(n_2) \le 2^m(2 \cdot 2^m - 2) + 2^{m+1} - 2 = 2^{2m+1} - 2^m \cdot 2 + 2^{m+1} - 2 = 2^{2m+1} - 2 < 2^{2m+1}.
\]
Thus, $f(n_2)$ can be expressed in the form $k2^m$ where $k < 2^{m+1}$. By Lemma~\ref{practical_numbers2}, since $2^m$ is practical and $k \le 2(2^m)$, $f(n_2)=k2^m$ is a practical number. This completes the proof for this case.

\vskip 5pt\noindent {\tt Case 3:} $n_1 = 0$ and $n_2 = 1$.

Since $n_1 = 0$ and $n_2 = 1$ are roots modulo $2^m$, we have $f(0) = c \equiv 0 \pmod{2^m}$ and $f(1) = 1 + b + c \equiv 0 \pmod{2^m}$. The first congruence implies that $c$ is a multiple of $2^m$. The second congruence implies that $b \equiv -1 \pmod{2^m}$, which means $b = r \cdot 2^m - 1$ for some positive integer $r$. 

By the definition of $m$, we have $\max\{b, c\} < 2^{m+1}$. Since $b$ and $c$ are positive integers, the inequality $c < 2^{m+1}$ implies that $c=2^m$. Similarly, the inequality $b < 2^{m+1}$ implies that $r$ is either $1$ or $2$. We further analyze both subcases:
\begin{itemize}
    \item \vskip 5pt\noindent {\tt Subcase 3a:} $b = 2^m - 1$ and $c = 2^m$. Here, $\max\{b, c\} = 2^m$. We choose $n = 2^m \le \max\{b, c\}$. Substituting this into the polynomial yields:
    \[
    f(2^m) = (2^m)^2 + (2^m - 1)2^m + 2^m = 2 \cdot 2^{2m} = 2^{2m+1}.
    \]
    $2^{2m+1}$ is always a practical number.
    
    \item \vskip 5pt\noindent {\tt Subcase 3b:} $b = 2^{m+1} - 1$ and $c = 2^m$. We have $\max\{b, c\} = 2^{m+1} - 1$. We choose $n = 2^m \le \max\{b, c\}$. Substituting this into the polynomial yields:
    \[
    f(2^m) = (2^m)^2 + (2^{m+1} - 1)2^m + 2^m = 3 \cdot 2^{2m}.
    \]
    By Lemma~\ref{practical_numbers2}, since $2^{2m}$ is a practical number and $1 \le 3 \le 2(2^{2m})$, the product $3 \cdot 2^{2m}$ is also a practical number.
\end{itemize}

Overall, there is an integer $n$ with $1 < n \leq \max\{b,c\}$ such that $n^2 + bn +c$ is practical in all the three cases.
\end{proof}

In this paper, we have provided a short proof for the second part of a conjecture proposed by Wang and Sun \cite{Wang} regarding the quadratic representations of practical numbers. Combined with the work of Somu, Li, and Kukla \cite{Somu}, the conjecture is now fully resolved.

An interesting direction for future research would be to investigate whether similar localized bounds can be established for practical numbers represented by higher-degree polynomial forms.

\vskip20pt\noindent {\bf Acknowledgements.} The author thanks Sai Teja Somu and Boduen Wang for reviewing the manuscript and providing insightful suggestions.


\begin{thebibliography}{1}\footnotesize

\bibitem{Melfi} G. Melfi, On two conjectures about practical numbers, {\it J. Number Theory} {\bf 56} (1996), 205--210.

\bibitem{Niven} I. Niven, H. S. Zuckerman, and H. L. Montgomery, {\it An Introduction to the Theory of Numbers}, 5th ed., John Wiley \& Sons, New York, 1991.

\bibitem{Somu} S. T. Somu, T. H. S. Li, and A. Kukla, On some results on practical numbers, {\it Integers} {\bf 23} (2023), \#A68.

\bibitem{Srinivasan} A. K. Srinivasan, Practical numbers, {\it Current Sci.} {\bf 17} (1948), 179--180.

\bibitem{Stewart} B. M. Stewart, Sums of distinct divisors, {\it Amer. J. Math.} {\bf 76} (1954), 779--785.

\bibitem{Wang} L.-Y. Wang and Z.-W. Sun, On practical numbers of some special forms, {\it Houston J. Math.} {\bf 48} (2022), 241--247.

\end{thebibliography}
\end{document}